\documentclass[12pt,a4paper]{amsart}
\usepackage{amsfonts, amsmath, amssymb, amscd, amsthm, graphicx}
\usepackage[pdftex]{hyperref}
\makeatletter
\@addtoreset{equation}{section} 

\makeatother

\newtheorem{thm}{Theorem}[section]
\newtheorem{cor}[thm]{Corollary}
\newtheorem{lemma}[thm]{Lemma}
\newtheorem{prop}[thm]{Proposition}

\theoremstyle{definition}

\theoremstyle{remark}
\newtheorem{rem}[thm]{Remark}
\newtheorem{ex}[thm]{Example}

\newcommand{\Z}{\mathbb{Z}}
\newcommand{\N}{\mathbb{N}}
\newcommand{\PT}{\mathcal{PT}}
\newcommand{\vr}{{\mathcal{VR}}}
\newcommand{\avr}{{\mathcal{AVR}}}
\newcommand{\ga}{\Gamma(G,{\mathcal A})}

\begin{document}
\title{Conjugacy in normal subgroups of hyperbolic groups}

\author{Armando Martino}
\author{Ashot Minasyan}
\address{School of Mathematics,
University of Southampton, Highfield, Sout\-hamp\-ton, SO17 1BJ, United
Kingdom.}
\email{A.Martino@soton.ac.uk} \email{aminasyan@gmail.com}

\begin{abstract} Let $N$ be a finitely generated normal subgroup of a Gromov hyperbolic group $G$.
We establish criteria for $N$ to have solvable conjugacy problem and be conjugacy separable in
terms of the corresponding properties of $G/N$. We show that the hyperbolic group
from F. Haglund's and D. Wise's version of Rips's construction
is hereditarily conjugacy separable. We then use this construction to produce
first examples of finitely generated and finitely presented conjugacy separable groups that contain
non-(conjugacy separable) subgroups of finite index.
\end{abstract}


\maketitle

\section{Introduction}

One of the most intuitive ways to study an infinite (discrete) group, $G$, is to look at its finite quotients.
However, in general, one loses information about $G$ if one restricts attention to these finite quotients, and
so arises the definition of {\em residual finiteness}, where $G$ is said to be residually finite if for every
$x \neq y \in G$ there is homomorphism $\psi:G \to Q$, where $Q$ is a finite group,
such that $\psi(x) \neq \psi(y)$ in $Q$.

To use this method more systematically, one introduces the \textit{profinite topology} $\PT(G)$ on $G$. This
is the topology whose basic open sets are cosets of finite index normal subgroups in $G$. It is not
difficult to check that the group operations are continuous with respect to this topology, thus $G$,
equipped with $\PT(G)$, is a topological group. Residual finiteness for $G$ then becomes the
statement that this topology is Hausdorff; in fact, it is equivalent to the statement that
singletons are closed.

A subset $A \subseteq G$ is said to be \textit{separable in} $G$ if $A$ is closed in $\PT(G)$.
Intuitively, this means that $A$ can be recognized by looking at finite quotients of $G$.
A group is called {\em subgroup separable} (or LERF) if every finitely generated subgroup is closed in
$\PT(G)$. Similarly, a group is called {\em conjugacy separable} if for every element $g \in G$,
its conjugacy class $g^G:= \{h^{-1}gh\,|\, h\in G\} \subseteq G$ is closed
in $\PT(G)$.
Equivalently, $G$ is conjugacy separable if and only if for any two non-conjugate
elements $x,y \in G$ there exists a homomorphism $\psi$ from $G$ to a finite
group $Q$ such that $\psi(x)$ is not conjugate to $\psi(y)$ in $Q$.

One can formalize the sense in which subsets closed in $\PT(G)$ are ``recognizable", and this was
done classically by A. Mal'cev \cite{Malcev}, who proved that a finitely presented residually finite
group has solvable word problem. Mal'cev \cite{Malcev} also proved that a finitely presented
conjugacy separable group has solvable conjugacy problem. (In fact, one can show that any
recursively enumerable subset of a finitely
presented group $G$, which is closed in $\PT(G)$, has decidable membership problem.)
Thus residual finiteness and conjugacy separability can be viewed as natural
``profinite analogues'' of the solvability of word and conjugacy problems respectively.

This analogy between decision problems and closure properties
can be seen as part of the motivation for this paper. Namely,  it is
well known that if a (finitely presented) group has solvable word problem, then so does any finite index
subgroup and any finite extension. Similarly, for a residually finite group, any finite index subgroup
and any finite extension is also residually finite. However, one cannot say the same about the
conjugacy problem: Collins and Miller \cite{Col-Mil} proved that there exists a finitely presented
group with solvable conjugacy problem, with a finite index subgroup having unsolvable conjugacy problem.
In the same paper \cite{Col-Mil}
(see also \cite{Gor-Kirk}), they also construct a finitely presented group with solvable conjugacy
problem and its finite extension with unsolvable conjugacy problem. The purpose of this paper is
to provide a profinite analogue of the first of these. Specifically, our main theorem is

\begin{thm}\label{thm:fp_ex} For every integer $m \ge 2$ there exists a finitely presented
conjugacy separable  group $T$ that contains a non-(conjugacy separable) subgroup $S$ of index $m$. Moreover,
$T$ can be chosen in such a way that
\begin{itemize}
	\item $T$ is a subgroup of some right angled Artin group;
	\item both $T$ and $S$ have solvable conjugacy problem.
\end{itemize}
\end{thm}

So while we would like conjugacy separability to pass to finite index subgroups, these examples show that it does not. For this reason, in \cite{Chag-Zal} S. Chagas and P. Zalesskii defined
a group $G$ to be \textit{hereditarily conjugacy separable} if every finite index subgroup of $G$ is
conjugacy separable, and constructed the first (infinitely generated) example of a
conjugacy separable but not hereditarily conjugacy separable group. The concept of
hereditary conjugacy separability is essential in the current paper. We believe that this concept
is stronger (according to the theorem above) and more useful than simply conjugacy separability,
in view of many applications, discovered in \cite{Min-RAAG}.

Let us also mention that an example of a finitely generated (but not finitely presented)
non-(conjugacy separable) group $G$,
containing a conjugacy separable subgroup $H$ of index $2$, was constructed much earlier
by A. Goryaga in \cite{Gor}. As it can be easily seen from Goryaga's argument,
the conjugacy problem is unsolvable in this group $G$.

In order to briefly describe the strategy of our proof and the structure of the paper,
let us start with the following

\begin{thm}\label{thm:dec_w_p<->conj_probl} Let $H$ be a finitely generated torsion-free
normal subgroup of a hyperbolic group $G$.
Then $H$ has solvable conjugacy problem if and only if $G/H$ has solvable word problem.
\end{thm}

The above statement was probably known to the experts before this work.
In one direction, it can be compared, for instance, with the result of
M. Bridson \cite{Bridson-conj_semihyp}, claiming that a normal subgroup $N$ of a bicombable group $G$ has solvable conjugacy problem, provided
the generalized word problem is solvable in $G/N$;
in the other direction it follows from \cite[III.$\Gamma$, Lemma 5.18]{B-H}. We give a proof of
Theorem \ref{thm:dec_w_p<->conj_probl} in Section \ref{sec:conj_prob}.
We include it here mainly to motivate Theorem \ref{thm:sep<->conj_sep} below, in which
the word and conjugacy problems are replaced by the corresponding properties
of the profinite topology.


\begin{thm}\label{thm:sep<->conj_sep} Suppose that $G$ is a torsion-free
hereditarily conjugacy separable hyperbolic group and $H \lhd G$ is a finitely generated normal subgroup.
Then $H$ is conjugacy separable if and only if $G/H$ is residually finite.
\end{thm}

As one can see, in Theorem \ref{thm:sep<->conj_sep} we had to impose
an additional assumption demanding that $G$ be hereditarily conjugacy separable.
Presently, it is not known whether there exist non-(conjugacy separable) or non-(residually finite)
hyperbolic groups. But if, for instance, there were a hyperbolic group $G$ that is
not residually finite, then the trivial
subgroup $H:=\{1\}$ would be conjugacy separable but $G/H \cong G$ would not be residually finite.
Thus the statement of Theorem \ref{thm:sep<->conj_sep} would be false without that additional assumption.

In Section \ref{sec:Rips} we suggest a generic method for constructing a hyperbolic group $G$
and its finitely generated normal subgroup $H$ satisfying the assumptions of
Theorem \ref{thm:sep<->conj_sep}. This method is based on the modification
of Rips's construction, presented by F. Haglund and D. Wise in \cite{H-W_1}, and on the result
that right angled Artin groups are hereditarily conjugacy separable, which was established in \cite{Min-RAAG}.
We apply this hereditarily conjugacy separable version of Rips's construction
together with Theorem \ref{thm:sep<->conj_sep}  in Section \ref{sec:main_constr},
to produce examples of finitely generated conjugacy separable groups which contain
non-(conjugacy separable) finite index subgroups.

In Section \ref{sec:fibre_cs} we obtain similar criteria for conjugacy separability of
fibre products, associated to normal subgroups of hyperbolic groups.
And in Section \ref{sec:f_p_ex} we construct the examples
demonstrating our main Theorem~\ref{thm:fp_ex}.

We note that the idea to employ Rips's construction for obtaining exotic examples of groups is an entirely familiar one,
having been used in \cite{Wise-incoherent}, \cite{B-G} and \cite{Bridson-direct}, for example.

\medskip
{\bf Acknowledgements.} We would like to thank M. Belolipetsky, M. Bridson,
F. Haglund and D. Wise for enlightening discussions.

\section{Preliminaries}
Let $(\mathcal{X},d)$ be a geodesic metric space and let $\delta$ be a non-negative real number.
A geodesic triangle $\Delta$ in $\mathcal X$ is said to be $\delta$-\textit{slim},
if each of its sides is contained in the closed $\delta$-neighborhood of the union of the two other
sides. The space $\mathcal X$ is called \textit{Gromov hyperbolic} if there exists
$\delta \ge 0$ such that every geodesic triangle $\Delta$ in $\mathcal X$ is $\delta$-slim.
A subset $Q \subseteq \mathcal X$ is said to be
$\varepsilon$-{\it quasiconvex} (for some $\varepsilon \ge 0$),
if any geodesic connecting two elements from $Q$
belongs to a closed $\varepsilon$-neighborhood of $Q$ in $\mathcal X$.

Given a group $G$, generated by a finite symmetrized (i.e., ${\mathcal A}={\mathcal A}^{-1}$) set
of elements $\mathcal A \subset G$, the corresponding Cayley graph $\ga$ can be equipped with the natural simplicial metric, making it a proper geodesic metric space. The group $G$
is (\textit{Gromov}) \textit{hyperbolic} if its Cayley graph $\ga$ is a Gromov hyperbolic metric space.
A subset $Q \subseteq G$ is called \textit{quasiconvex} if there is $\varepsilon \ge 0$ such that
$Q$ is $\varepsilon$-quasiconvex, when regarded as a subset of $\ga$.

Hyperbolic groups became a major subject of study in Geometric Group Theory since they were
introduced by M. Gromov in \cite{Gromov}. Hyperbolicity of a given group $G$  does not depend on the choice of a
particular finite generating set $\mathcal A$ of $G$. In the case when $G$ is hyperbolic,
quasiconvexity  of some subset $Q \subseteq G$ is also independent of $\mathcal A$.
For this and other basic properties of hyperbolic groups the reader is
referred to \cite{Ghys-Harpe} and \cite{Mihalik}.

It is well known that hyperbolic groups have solvable word and conjugacy problems
(see, for example, \cite[III.$\Gamma$.2.8]{B-H}).

Assume, now, that $G$ is a hyperbolic group, $H \le G$ is a subgroup and $S \subseteq G$ is a subset.
The \textit{centralizer} $C_H(S)$,
of $S$ in $H$, is the subgroup $\{h \in H\,|\, hs=sh, \forall\,s \in S\} \le G$.
It is a basic fact that for a hyperbolic group, $G$, and
 each infinite order element $g \in G$,
the cyclic subgroup $\langle g \rangle$ has finite index in the centralizer $C_G(g)$ (see \cite[8.3.34]{Ghys-Harpe}).

The subgroup $H \le G$ is called \textit{elementary} if it contains a cyclic subgroup of finite index.
Every infinite order element $g \in G$
belongs to a unique \textit{maximal elementary subgroup}
$E_G(g) \le G$ (see, for instance, \cite[Lemma 1.16]{Olsh-G-sbgps}).
In particular, $C_G(g) \subseteq E_G(g)$.

As A.~Olshanskii showed in  \cite[Prop. 1]{Olsh-G-sbgps},
for every non-elementary subgroup $H \le G$
there is a unique \textit{maximal finite subgroup} $E_G(H)$ \textit{normalized by} $H$ in $G$.
More precisely, $E_G(H)=\bigcap_{g \in H^0} E_G(g)$, where $H^0$ denotes the subset of all infinite
order elements in $H$.  It is not difficult to see that for a non-elementary subgroup $H \le G$
one has $C_G(H) \le E_G(H)$.

Let $\Gamma$ be a finite simplicial graph, and
let $\mathcal V$ and $\mathcal E$ be the sets of vertices and edges of $\Gamma$ respectively.
The \textit{right angled Artin group} $G$, associated to $\Gamma$, is given by the presentation

$$G:=\langle \mathcal V \,\|\,uv=vu, \,
\mbox{whenever } u,v \in \mathcal{V} \mbox{ and }  (u,v) \in \mathcal E \rangle.
$$ 

Let $G$ be a group and $H \le G$. Then $H$ is said to be a \textit{virtual retract} of $G$ if there
is a finite index subgroup $K \le G$ such that $H \le K$ and $H$ is a \textit{retract} of $K$
(that is, there exists an endomorphism $\rho:K \to K$ satisfying $\rho(K)=H$ and $\rho \circ \rho=\rho$).

The following two classes of groups were introduced in \cite{Min-RAAG}.
The first class $\vr$ consists of all groups that are virtual retracts of right angled Artin groups.
A group $G$ belongs to the second class  $\avr$, by definition, if $G$ contains a finite index
subgroup $H \le G$ with $H \in \vr$.

\section{Conjugacy problem for subgroups}\label{sec:conj_prob}

\begin{prop}\label{prop:memb->conj}
Let $G$ be a hyperbolic group and $H$ a finitely generated torsion-free subgroup of $G$.
If $H$ has solvable membership problem in $G$, then it also has solvable conjugacy problem.
\end{prop}

\begin{proof} Consider arbitrary $x, y \in H$; since the word problem in $G$ (and hence in $H$)
is solvable, without loss of generality we can assume that $x \neq 1$.
As $G$ is hyperbolic, we can decide if $x,y$ are
conjugate in $G$. If they are not,
then neither are they conjugate in $H$. If they are conjugate in $G$ we may, by enumeration, find a
conjugator $g \in G$ such
that $x^g:=g^{-1} x g=y$ (note that this relies on $G$ being finitely generated and having
solvable word
problem). It is then clear
that $x,y$ are conjugate in $H$ if and only if $g \in C_G(x)H$. However, since $H$ is torsion
free, $C_G(x)$ must be virtually cyclic.
Therefore, the subgroup generated by $x$ has finite index in $C_G(x)$, and, furthermore, it
is possible to algorithmically find a finite set of coset
representatives for $\langle x \rangle$ in $C_G(x)$(see \cite[Prop. 4.11]{B-M-V}).
We shall call these $h_1, \ldots , h_k$.
Therefore, $x$ and $y$ are conjugate in $H$ if and only if $h_i^{-1} g \in H$ for some $i$.
Since the hypothesis allows us to decide this,
we have solved the conjugacy problem in $H$.
\end{proof}

\begin{rem} The assumption that $H$ is finitely generated in  Proposition \ref{prop:memb->conj} is not really needed:
the set of words representing elements from $H$ is recursive because
there is an algorithm that solves the membership problem for $H$ in $G$.
\end{rem}

\begin{prop}\label{prop:conj->memb}
Let $G$ be a hyperbolic group and $H$ a finitely generated subgroup of $G$.
If $H$ is normal and has solvable conjugacy problem, then it also has solvable membership problem in $G$.
\end{prop}

\begin{proof}

If $H$ is finite, the claim follows from the solvability of the word problem in $G$.
So we may assume that $H$ is infinite. In this case there exists an element $x \in H$ of infinite
order (by \cite[8.3.36]{Ghys-Harpe}).
Consequently, $\langle x \rangle$ has finite index in $C_G(x)$ and so $C_H(x)$ has finite
index in $C_G(x)$.
Thus there exists a finite collection of elements $1=f_1, f_2, \ldots ,f_k$ in $G$, which are the
right coset representatives of $C_H(x)$ in $C_G(x)$.
Since the element $x$ is fixed throughout, we may assume that these are given (as words in the
generators of $G$).

Consider some $g \in G$; we need to decide whether or not it lies in $H$.
Since $H$ is normal, $x^g \in H$. Clearly, if $x^g$ is not conjugate to $x$ {\em in} $H$,
then $g \notin H$ and we are done. Otherwise,
enumerating all elements of $H$, we will be able to find an $h \in H$ such that $x^g = x^h$.
Hence $gh^{-1} \in C_G(x)$ and $g \in H$ if and only if $gh^{-1} \in C_H(x)$.
Now we know that $g h^{-1} f_i^{-1} \in H$ for some $i$. As $H$ is finitely
generated and $G$ has solvable word problem, we
may enumerate all elements of $H$, and for each $i$ check whether  $g h^{-1} f_i^{-1}$
is equal to this element of $H$. By
construction, this process is guaranteed to terminate after finitely many steps, and $g \in H$
if and only if the process
tells us that $g h^{-1} f_1^{-1} \in H$.
\end{proof}

\begin{proof}[Proof of Theorem \ref{thm:dec_w_p<->conj_probl}] It is easy to see that
for a normal subgroup $H$ of a group $G$, the membership problem to $H$ in $G$
is equivalent to the word problem in the quotient $G/H$. Therefore, the claim of
Theorem \ref{thm:dec_w_p<->conj_probl} immediately follows from Propositions \ref{prop:conj->memb} and
\ref{prop:memb->conj}.
\end{proof}

\section{Conjugacy separable subgroups}\label{sec:conj_sep}
As Mal'cev proved in \cite{Malcev}, if $H$ is a finitely generated separable subgroup
of a finitely presented group $G$, then
the membership problem for $H$ in $G$ is solvable.
Therefore one can see that the next fact is a ``profinite analogue'' of Proposition \ref{prop:memb->conj}.

\begin{prop}\label{prop:sep->conj_sep} Let $H$ be a torsion-free subgroup
of a hyperbolic group $G$. If $G$ is hereditarily conjugacy separable and
$H$ is separable in $G$ then $H$
is conjugacy separable.
\end{prop}

\begin{proof} Consider any elements $x,y \in H$ that are not conjugate in $H$.
Without loss of
generality, we may assume that $x \neq 1$. Then $x$ must have infinite order, by our hypothesis, and
therefore $|C_G(x):\langle x \rangle|<\infty$. Let $L:=C_H(x)=C_G(x) \cap H$,
then $x \in L$ and one can find elements $z_1,\dots,z_k \in C_G(x)\setminus H$
such that $C_G(x)=L \sqcup \bigsqcup_{i=1}^k z_iL$.

Since $H$ is separable in $G$, there exists a finite index subgroup $K \le G$ such that $H \le K$
and $z_i \notin K$ for every $i=1,\dots,k$. Observe that, by construction,
the centralizer $C_K(x)$, of $x$ in $K$, is contained in $H$ (in fact, it is equal to $L$).

Assume, first, that $y=gxg^{-1}$ for some $g \in K$. Since $x$, $y$ are not conjugate in $H$ we know that
$g \notin H$. Also, $H$ is closed in $\PT(K)$ because $H$ is closed in $\PT(G)$ and $K \le G$.
Hence, there is a finite index subgroup $M \le K$ such that $H \le M$ and $g \notin M$. We claim that
$x$ is not conjugate to $y$ in $M$. Indeed, otherwise, there would exist $u \in M$
with $uxu^{-1}=y=gxg^{-1}$, implying that $u^{-1}g \in C_K(x) \le H$, i.e., $g \in uH \subseteq M$, which
contradicts the choice of $M$.

Thus we proved that there is a finite index subgroup
$N \le G$ such that $H \le N$ and $x$ is not conjugate to $y$
in $N$. Using the conjugacy separability of $N$, we can
find a finite group $Q$ and a homomorphism $\varphi:N \to Q$ such that $\varphi(x)$
is not conjugate to $\varphi(y)$ in $Q$. And since $H \le N$, the restriction of $\varphi$ to $H$
gives a finite quotient of $H$, in which the images of $x$ and $y$ are not conjugate.
\end{proof}

Next comes a natural analogue of Proposition \ref{prop:conj->memb}. Observe that, unlike
the previous statement, it does not need to assume that $G$ is
hereditarily conjugacy separable, but requires the subgroup $H$ to be normal and finitely generated.

\begin{prop}\label{prop:conj_sep->sep_quot} Let $H$ be a finitely generated
non-elementary normal subgroup
of a hyperbolic group $G$, with $E_G(H)=\{1\}$. If $H$ is conjugacy separable then
$G/H$ is residually finite.
\end{prop}

The proof of Proposition \ref{prop:conj_sep->sep_quot} makes use of the three lemmas below.
The next statement is well-known and can be verified in a straightforward manner
(see, for example, \cite[Lemma 5.2]{Min-CC}).

\begin{lemma} \label{lem:out-gp} Assume $G$ is a group and $H \lhd G$ is a normal subgroup such that
$C_G(H)\subseteq H$. Then the quotient-group $G/H$ embeds into the outer automorphism group
$Out(H)$.
\end{lemma}

An automorphism $\psi$ of a group $G$ is said to be \textit{pointwise inner}, if for each $g \in G$,
$\psi(g)$ is conjugate to $g$ in $G$. It is not difficult to see that the set of all pointwise
inner automorphisms forms a normal subgroup $Aut_{p.i.}(G) \lhd Aut(G)$, and $Inn(G) \le Aut_{p.i.}(G)$.
The following criterion was found by E. Grossman in \cite[Thm. 1]{Grossman}:

\begin{lemma} \label{lem:Gros} Let $H$ be a finitely generated conjugacy separable group
with  $Aut_{p.i.}(H)=Inn(H)$. Then $Out(H)$ is residually finite.
\end{lemma}

The last ingredient is given by  \cite[Cor. 5.4]{Min-Osin}, where D. Osin  and the second author
showed that $Aut_{p.i.}(H)=Inn(H)$ for any non-elementary subgroup $H$ of a relatively hyperbolic group
$G$, provided $H$ contains at least one infinite order element that is not conjugate to an element
of a parabolic subgroup in $G$, and $E_G(H) = \{ 1 \}$. Since every hyperbolic group is relatively hyperbolic with respect to
the trivial subgroup (see \cite{Osin-rel_hyp}),
and an infinite subgroup of a hyperbolic group necessarily contains
an element of infinite order (\cite[8.3.36]{Ghys-Harpe}), we obtain

\begin{lemma} \label{lem:Min-Os} If $H$ is a non-elementary subgroup of a hyperbolic group $G$
with $E_G(H)=\{1\}$, then $Aut_{p.i.}(H)=Inn(H)$.
\end{lemma}

\begin{proof}[Proof of Proposition \ref{prop:conj_sep->sep_quot}] By Lemmas \ref{lem:Min-Os}
and \ref{lem:Gros}, $Out(H)$ is residually finite. And $G/H$ embeds in $Out(H)$ according
to Lemma \ref{lem:out-gp}, because $C_G(H)\subseteq E_G(H)=\{1\}$. Therefore,
$G/H$ is residually finite as well.
\end{proof}

\begin{rem} We do not know whether one can remove the assumption that $H$ is non-elementary
from Proposition \ref{prop:conj_sep->sep_quot},
because this is directly related to the well-known open question about the existence of
non-(residually finite) hyperbolic groups.
\end{rem}

\begin{rem} The condition $E_G(H)=\{1\}$ in Proposition \ref{prop:conj_sep->sep_quot} is equivalent to the condition $E_G(G)=\{1\}$.
In other words, it says that $G$ contains no non-trivial finite normal subgroups.
\end{rem}

Indeed, since $E_G(H)$ is the (only) maximal finite subgroup of $G$ normalized by $H$ and $H \lhd G$, it is easy to see
that $E_G(H) \lhd G$, hence $E_G(H) \le E_G(G)$. Evidently, $E_G(G) \le E_G(H)$, hence $E_G(H)=E_G(G)$.

We are now ready to prove Theorem \ref{thm:sep<->conj_sep}.
\begin{proof}[Proof of Theorem \ref{thm:sep<->conj_sep}] The sufficiency is an immediate
consequence of Proposition \ref{prop:sep->conj_sep}, because a normal subgroup is separable if and only if
the quotient by it is residually finite.

In order to prove the necessity, suppose that $H$ is conjugacy separable. If $H$ is virtually cyclic,
then $H$ is quasiconvex in $G$ (see \cite[Cor. 3.4]{Mihalik}). And since $H\lhd G$,
by  \cite[Prop. 3.9]{Mihalik},
$H$ is either finite or has finite index in $G$. In either of these two cases, $G/H$ is residually
finite: in the latter case this is obvious and in the former case this follows from the fact that
$G$ is residually finite (since it is conjugacy separable),
because any finite subset of a residually finite group  is separable.

Thus we can assume that $H$ is non-elementary. Observe that $E_G(H)=\{1\}$ because $G$ is torsion-free.
Consequently, $G/H$ is residually finite by Proposition \ref{prop:conj_sep->sep_quot}.
\end{proof}


\section{Hereditarily conjugacy separable Rips's construction}\label{sec:Rips}

Rips's construction, discovered by E. Rips \cite{Rips},
is a very useful tool, that turned out to be a rich source of counterexamples,
allowing one to find finitely generated subgroups
of word hyperbolic groups with exotic properties (e.g., subgroups that are finitely generated but not finitely
presented; finitely generated subgroups with undecidable membership problem; etc.).
Briefly speaking, for every finitely presented group $P$, Rips's construction produces
a hyperbolic group $G$ together with a finitely generated normal subgroup $N \lhd G$ such that
$G/N \cong P$. In fact, in addition to being hyperbolic, the group $G$ can be made to
enjoy other agreeable properties. For instance, in \cite{Wise-rips} Wise provided a modification of
Rips's construction, in which the group $G$ is residually finite. More recently,
Haglund and Wise \cite{H-W_1} suggested a different version of Rips's construction
(and proved its key properties), producing
$G$ as the fundamental group of a compact non-positively curved square
thin $\mathcal{VH}$-complex. In this case the group $G$ will be linear and will have separable
quasiconvex subgroups (see \cite{H-W_1}). Moreover, as we show below,
$G$ will also be hereditarily conjugacy separable.
In writing this paper we have become aware of the recent work of O. Cotton-Barratt and H. Wilton \cite{CB-W},
where it is shown that the residually finite version of Rips's construction, originally introduced
by Wise in \cite{Wise-rips}, is also conjugacy separable.

Thin $\mathcal{VH}$-complexes were introduced by Wise in \cite{Wise-polyg}.
For our purposes, we only need to know three
facts about them. The first fact is that if
$G$ is a fundamental group of compact thin $\mathcal{VH}$-complex $\mathcal X$, then it is word hyperbolic.
Indeed, as shown by Wise in \cite{Wise-polyg}, 
the universal cover
$\widetilde{\mathcal{X}}$ is a  ${\rm CAT}(0)$ space which contains no immersed flats. Hence
$\pi_1(\mathcal{X})$ is word hyperbolic by a theorem of M. Bridson \cite{Bridson}.
The second fact is that $\pi_1(\mathcal{X})$ is torsion-free (for instance, because it acts freely
on the locally compact ${\rm CAT}(0)$ space $\widetilde{\mathcal{X}}$, and any finite group
acting on such a space fixes at least one point  -- see \cite[II.2.8]{B-H}).
The third fact, proved by Haglund and Wise in \cite{H-W_1}, tells us that $G=\pi_1(\mathcal{X})$
belongs to the class $\avr$.


\begin{thm}\label{thm:Rips-h_c_s} Let $P$ be an arbitrary finitely presented group. Then there exist
a torsion-free word hyperbolic group $G$ and a finitely generated (normal) subgroup $N \lhd G$ such that
$G/N \cong P$. Moreover, such a group $G$ can be taken to satisfy all of the following conditions:
\begin{itemize}
	\item[(i)] $G$ is the fundamental group of a compact non-positively curved  thin $\mathcal{VH}$-complex;
	\item[(ii)] $G \in \avr$; 
	\item[(iii)] $G$ is hereditarily conjugacy separable.
\end{itemize}
\end{thm}

The above theorem is essentially due to Haglund and Wise (see \cite{H-W_1}).
Only the property (iii) is new,
but it follows from (ii) and the fact that torsion-free hyperbolic groups from the class $\avr$
are hereditarily conjugacy separable, which was proved by the second author in \cite[Cor. 9.11]{Min-RAAG}.

Finally, we note that for a given finite presentation of $P$, a finite presentation for
the group $G$ from Theorem \ref{thm:Rips-h_c_s} can be constructed explicitly (see
\cite[Thm. 10.1]{H-W_1}).


\section{Constructing non-hereditarily conjugacy separable groups}\label{sec:main_constr}
It is known that if a finitely presented group $Q$ has solvable word problem, then for any finite group
$F$, any extension $P$ of $F$ by $Q$ also has solvable word problem (cf. \cite[Lemma 4.7]{Miller-survey}).
This group $P$  will also be finitely presented, as is any (finitely presented)-by-(finitely presented)
group -- see \cite[Lemma 1]{Hall}. However,
residual finiteness of $Q$ is not always passed to $P$.
(The reader should recall and contrast with the fact that
residual finiteness {\em is} passed to arbitrary subgroups  and finite extensions.)

More precisely, our algorithm for constructing (non-hereditarily) conjugacy separable groups
takes as an input a short exact sequence of groups
\begin{equation} \label{eq:short_exact} \{1\} \to F \to P \to Q \to \{1\},
\end{equation}
where $F$ is a finite group, $Q$ is residually finite and $P$ is finitely presented but not
residually finite. The existence of such short exact sequences is non-trivial. We describe two
ways to construct them below.

As observed by J. Corson and T. Ratkovich
in \cite{Cors-Ratk}, if one has a short exact sequence $\{1\} \to M \to R \to Q \to \{1\}$ of groups,
where $M, Q$ are residually finite, $M$ is finitely generated and $R$ is not residually finite,
then one can find a finite quotient $F$ of $M$ and an extension $P$, of $F$ by $Q$, which is not residually finite.


In the current literature we were able to find two examples of non-(residually
finite) finitely presented (residually finite)-by-(residually finite) groups.

\begin{ex}\label{ex:1}
In \cite{Deligne}, P. Deligne proved that for every integer $n \ge 2$,  there is a
central extension $R$ of the infinite cyclic group $\langle a \rangle \cong\Z$ by the group
of integral symplectic matrices ${\rm Sp}(2n,\Z)$ ,
such that the element $a^2$ belongs to the kernel of every homomorphism from $R$ to a finite group
(more precisely, $R=\widetilde{{\rm Sp}(2n,\Z) }$
is the inverse image of ${\rm Sp}(2n,\Z)$ in the universal cover of ${\rm Sp}(2n,\mathbb{R})$).
Consequently, for every $m \in \N$, the group $R_m:=R/\langle a^m \rangle$ is a central
extension of the cyclic group $\Z / m\Z$ by ${\rm Sp}(2n,\Z)$.

The group ${\rm Sp}(2n,\Z)$ is finitely presented (since it is an arithmetic subgroup of the algebraic group
${\rm Sp}(2n,\mathbb{R})$ -- see \cite[Ch. 4.4, Thm. 4.2]{Plat-Rap})
and residually finite (as is any finitely generated
subgroup of a linear group, according to Mal'cev's theorem \cite{Mal'cev-matrix}).
Thus, the group $R_m$ is finitely presented and contains a
central cyclic subgroup $C_m$ of order $m$, such that $R_m/C_m \cong {\rm Sp}(2n,\Z)$.
And if $m \ge 3$, $R_m$ is not residually finite by the above theorem of Deligne.

The group $R_2$ may be residually finite, but in this case
$R_4$, mapping onto $R_2$, contains a finite index subgroup $P$ that avoids the generator $c$ of
the central subgroup $C_4$. By Deligne's result, $c^2 \in P$, hence we have a short
exact sequence $\{1\} \to \langle c^2 \rangle_2 \to P \to Q \to \{1\}$, where $Q$ is a
finite index subgroup of ${\rm Sp}(2n,\Z)$ (more precisely, $Q$ is the image of $P$ in ${\rm Sp}(2n,\Z)$).
Thus $Q$ is residually finite,
$P$ is a finitely presented extension of $\Z/2\Z$ by $Q$, and $P$ is not residually
finite (because $R_4$ is not).
\end{ex}

\begin{ex}\label{ex:2} More recently, P. Hewitt \cite{Hewitt} proved that there exists an extension $E$
of the free abelian group $A := \Z^3$ by ${\rm SL}(3,\Z)$ that is not residually finite
(however, no explicit constructions for $E$ are known so far).
It follows (see \cite{Cors-Ratk}), that for every sufficiently large $m \in \N$
there is a non-(residually finite) extension $E_m:=E/A^m$ of  $A/A^m \cong (\Z/m\Z)^3$ by ${\rm SL}(3,\Z)$.

It is well known that ${\rm SL}(3,\Z)$ is finitely presented and residually finite,
therefore the short exact sequence $\{1\} \to (\Z/m\Z)^3 \to E_m \to {\rm SL}(3,\Z) \to \{1\}$
enjoys the required properties.

Let $B_m:=A/A^m\cong (\Z/m\Z)^3$ denote the image of $A$ in $E_m$, and let $T_m:=C_{E_m}(B_m)$.
Then $B_m \le T_m$ and $|E_m:T_m|<\infty$ because $B_m$ is a finite normal subgroup of $E_m$.
Thus $T_m$ is a central extension of $B_m$ by a finite index subgroup of ${\rm SL}(3,\Z)$
(which is the image of $T_m$ under the epimorphism $E_m \to {\rm SL}(3,\Z)$).
As before, $T_m$ fails to be residually finite because it has finite index
in the non-(residually finite) group $E_m$. We can now argue
similarly to Example \ref{ex:1}, to produce, for every $q \ge 2$,
a central extension of the cyclic group
$C_q$, of order $q$, by a finite index subgroup of ${\rm SL}(3,\Z)$, which is not residually finite.

\end{ex}

{\bf Main construction.} Start with the short exact sequence \eqref{eq:short_exact} such that
$F$ is a finite group, $Q$ is residually finite and $P$ is finitely presented but not
residually finite. By Theorem \ref{thm:Rips-h_c_s} we can find a torsion-free hereditarily conjugacy
separable hyperbolic group $G$ and a finitely generated normal subgroup $N \lhd G$ such that $G/N \cong P$.
 Thus there is an epimorphism $\psi:G \to P$
with $\ker(\psi)=N$. Let $H \le G$ be the full preimage of $F$ under $\psi$.
We have the following commutative diagram:

\begin{equation}
\label{eq:cd}
\begin{CD}
  N    @>>> H @>>> G \\
@VVV        @VVV  @V{\psi}VV \\
\{1\} @>>> F    @>>> P @>>> Q
\end{CD}
\end{equation}

Then $H \lhd G$, $N \le H$ and $H$ is finitely
generated (since $H/N \cong F$ is finite).
Observe that $G/H \cong P/F \cong Q$ is residually finite, hence $H$ is conjugacy separable
according to Theorem \ref{thm:sep<->conj_sep}. On the other hand, $G/N \cong P$ is not residually finite,
and Theorem \ref{thm:sep<->conj_sep} implies that $N$ is not conjugacy separable.
And since $N$ has finite index in $H$, we see that the group
$H$ is conjugacy separable, but not hereditarily conjugacy separable.

\begin{rem}\label{rem:1} The group $G$ in the above construction can be chosen from the class $\vr$.
Therefore the (non-hereditarily) conjugacy separable group $H$ will be a subgroup
of some right angled Artin group $A$.
\end{rem}

Indeed, by Theorem \ref{thm:Rips-h_c_s}, the group $G$ from the main construction
belongs to the class $\avr$. Hence there is a finite index subgroup $G_1 \le G$
such that $G_1 \in \vr$. Note that $G_1$ is hereditarily conjugacy separable, torsion-free and hyperbolic,
because all of these properties are inherited by finite index subgroups.
Denote $N_1:=N \cap G_1$,  $H_1:=H \cap G_1$, $P_1:=\psi(G_1)$,
$F_1:=P_1 \cap F \lhd P_1$, and let $Q_1$ be the image of $P_1$ in $Q$.
Clearly  $|P:P_1|<\infty$, therefore $P_1$ cannot be residually finite.
And since $G_1/N_1 \cong P_1$, $G_1/H_1 \cong Q_1 \le Q$, Theorem \ref{thm:sep<->conj_sep} yields that
$H_1$ is conjugacy separable and $N_1$ is not conjugacy separable. Finally, we see that
$|H_1:N_1| =|F_1| \le|F| <\infty$.

\begin{rem}\label{rem:solv_conj} The groups $N$ and $H$ from the main construction have solvable
conjugacy problem.
\end{rem}

This is an immediate consequence of Theorem \ref{thm:dec_w_p<->conj_probl}, because
$Q\cong G/H$ and $P \cong G/N$ both have solvable word problem. Indeed, $Q$ is finitely presented and
residually finite, hence the word problem in $Q$ is solvable by Mal'cev's result \cite{Malcev}.
Therefore, $P$, being an extension of the finite group $F \cong H/N$ by $Q$, has solvable word problem
as well.

\section{Conjugacy separable fibre products}\label{sec:fibre_cs}

Let $G$ be a group. To every normal subgroup $N \lhd G$ one can associate the \textit{fibre product}
subgroup $T_N \le G \times G$, defined by $T_N:=\{(g_1,g_2) \in G\times G \,|\, \psi(g_1)=\psi(g_2)\}$,
where $\psi: G \to G/N$ is the natural epimorphism. It is not difficult to see that $T_N$
is the product of $N \times N \lhd G \times G$ with the diagonal subgroup
$\{(g_1,g_1) \,|\, g_1\in G\} \le G\times G$.
The following remarkable
``$1$-$2$-$3$ Theorem'', discovered
by G. Baumslag, M. Bridson, C. Miller and H. Short in \cite{B-B-M-S},
provides sufficient conditions for finite presentability of the fibre product:

\begin{lemma} \label{lem:1-2-3}
Suppose that $N$ is a normal subgroup of a group $G$, $Q:=G/N$ and $T_N \le G \times G$
is the fibre product associated to $N$.
If $N$ is finitely generated, $G$ is finitely presented and $Q$ is of type $F_3$, then
$T_N$ is finitely presented.
\end{lemma}

(Recall that a group is said to be \textit{of type} $F_n$ if it has an Eilenberg-Maclane CW complex
with only finitely many $k$-cells for each $k \le n$.)

The above criterion together with the hereditarily conjugacy separable version of Rips's construction
can be used to produce finitely presented (non-hereditarily) conjugacy separable groups. However,
before doing this, we need to establish analogues of Propositions \ref{prop:sep->conj_sep} and
\ref{prop:conj_sep->sep_quot} for fibre products.


If $G$ is a group, $H \le G$ and $x \in G$, the $H$-\emph{conjugacy class} of $x$ in $G$ is,
by definition, the
set $x^H:=\{h^{-1}xh \, | \, h \in H\}$.

\begin{lemma}\label{lem:c_d_sbgps} Suppose $H$ is a finite index subgroup of a group $K$ and
$x \in H$. If $x^H$ is closed in $\PT(H)$ then $x^K$ is closed in $\PT(K)$.
\end{lemma}

\begin{proof} Choose $g_1,\dots,g_k \in K$ so that $K=\bigsqcup_{i=1}^k Hg_i$.
Since $|K:H|<\infty$, any subset of $H$ which is closed in $\PT(H)$, is also closed in
$\PT(K)$ (because any subgroup of finite index in $H$ also has finite index in $K$).
Therefore, $x^H$ is closed in $\PT(K)$. Consequently, $x^K=\bigcup_{i=1}^k g_i^{-1} x^H g_i$ is closed
in $\PT(K)$ as a finite union of closed sets.
\end{proof}

\begin{lemma}\label{lem:dir_prod-hcs} Let $G_1$ and $G_2$ be hereditarily conjugacy separable groups. Then their direct product
$G:=G_1 \times G_2$ is also hereditarily conjugacy separable.
\end{lemma}

\begin{proof} Consider any finite index subgroup $K \le G$ and any $x=(x_1,x_2) \in K$.
Then there exist finite index normal subgroups $N_i \lhd G_i$, $i=1,2$, such
that $N:=N_1\times N_2 \le K$. Set $H:=\langle x \rangle N \le K$,
then $x \in H$, and $H$ has finite index in $K$ and in $G$.

Take any element $y=(y_1,y_2) \in H$ such that $y \notin x^H$. Then there is $j \in \{1,2\}$
such that $y_j \notin x_j^{N_j}$. Observe that $x_j^{N_j}=x_j^{H_j}$, where $H_j:=\langle x_j \rangle N_j$
is the image of $H$ under the canonical projection $\rho_j:G \to G_j$. Thus  $x_j,y_j \in H_j$
and $y_j \notin x_j^{H_j}$.  Now, since $|G_j:H_j|\le |G_j:N_j|<\infty$,
$H_j$ is conjugacy separable by the assumptions. Hence there is a finite
group $Q$ and a  homomorphism $\psi:H_j\to Q$
such that $\psi(y_j) \notin \psi(x_j)^Q$. Define the homomorphism $\varphi:H \to Q$ by
$\varphi:=\psi \circ \rho_j$. Evidently, $\varphi(y)=\psi(y_j) \notin  \psi(x_j)^Q=\varphi(x)^Q$.

Thus we have shown that $x^H$ is closed in $\PT(H)$. Therefore, by Lemma \ref{lem:c_d_sbgps},
$x^K$ is closed in $\PT(K)$ for every $x \in K$. And we can conclude that
$K$ is conjugacy separable, as required.
\end{proof}

The next statement is proved in \cite[Cor. 11.2]{Min-RAAG}.

\begin{lemma}\label{lem:h_c_s->c_s_for_sbgps} Let $G$ be a hereditarily conjugacy separable group.
Suppose that $H$ is a subgroup of $G$ such that the double coset $C_G(h)H$ is separable in $G$ for every
$h \in H$. Then $H$ is conjugacy separable.
\end{lemma}

A group $G$ is said to be \textit{cyclic subgroup separable} if for every $g \in G$
the cyclic subgroup $\langle g \rangle \le G$ is closed in $\PT(G)$.

\begin{prop}\label{prop:fibre_sep->conj_sep} Let $H$ be a normal subgroup
of a torsion-free hyperbolic group $G$. If $G$ is hereditarily conjugacy separable and
$G/H$ is cyclic subgroup separable, then the corresponding fibre product $T_H \le G \times G$
is conjugacy separable.
\end{prop}

\begin{proof} Denote $Q:=G/H$, let
$\psi:G \to Q$ be the natural epimorphism, and let $\eta: G \times G \to Q \times Q$ be the homomorphism
defined by $\eta(g_1,g_2):=(\psi(g_1), \psi(g_2))$ for all $(g_1,g_2) \in G \times G$.
Consider any element $x=(x_1,x_2) \in T_H$. We will show that the double coset $C_{G \times G}(x)T_H$
is separable in $G\times G$, and then Lemma \ref{lem:h_c_s->c_s_for_sbgps} will allow us to conclude
that $T_H$ is conjugacy separable.

If $x_1=1$ (or $x_2=1$) in $G$, then
$(G,1) \le C_G(x)$ ($(1,G) \le C_G(x)$), and since $T_H$ contains the diagonal subgroup of $G \times G$,
we have $C_G(x)T_H=G \times G$. Thus, in this case the double coset $C_G(x) T_H$ is separable
in $G \times G$.

So, we can suppose that both $x_1$ and $x_2$ are infinite order elements in $G$. Then
$|C_G(x_i):\langle x_i \rangle|<\infty$ for $i=1,2$, and since
$C_{G \times G}(x)=C_G(x_1) \times C_G(x_2)$ in $G \times G$, we see that
$|C_{G \times G}(x): \langle x_1 \rangle \times \langle x_2 \rangle|<\infty$.
Consequently, there exist elements $z_1,\dots,z_n \in C_{G \times G}(x)$ such that
$C_{G \times G}(x)=\bigsqcup_{i=1}^n z_i \left( \langle x_1 \rangle \times \langle x_2 \rangle\right)$.

Note that
$\langle x_1 \rangle \times \langle x_2 \rangle=\langle (x_1,1) \rangle \langle(x_1,x_2)\rangle$
in $G \times G$. Therefore, since $x=(x_1,x_2) \in T_H$, we have
$\left(\langle x_1 \rangle \times \langle x_2 \rangle \right) T_H=\langle (x_1,1) \rangle T_H$.
It is easy to see that $\eta(T_H)=D$, where $D:=\{(q,q) \,|\, q \in Q\}$ is the diagonal
subgroup of $Q \times Q$, and $T_H=\eta^{-1}(D)$ is the full preimage of $D$
in $G \times G$. Hence $\langle (x_1,1) \rangle T_H=
\eta^{-1}\left(\langle a \rangle D\right)$, where $a:=\eta((x_1,1))=(a_1,1) \in Q \times Q$,
$a_1:=\psi(x_1)$.

Observe that $(q_1,q_2) \in \langle a \rangle D$ in $Q\times Q$ if and only if
$q_1q_2^{-1} \in \langle a_1 \rangle$ in $Q$. Therefore, if
$(q_1,q_2) \notin \langle a \rangle D$, we can use the assumption that $Q$ is cyclic subgroup
separable to find a finite index normal subgroup $M \lhd Q$ with
$q_1q_2^{-1} \notin \langle a_1 \rangle M$. Hence
$(q_1,q_2) \notin \langle a \rangle D (M \times M)$ in $Q \times Q$, and
$M \times M$ has finite index in $Q \times Q$. Thus we have shown that the double coset
$\langle a \rangle D$ is closed in $\PT(Q\times Q)$. Since $\eta:G \times G \to Q \times Q$
is a continuous map (with respect to the corresponding profinite topologies), we can conclude that
$\langle (x_1,1) \rangle T_H=\eta^{-1}(\langle a \rangle D)$ is closed in $\PT(G \times G)$.

In $G \times G$ we have
$C_{G \times G}(x) T_H=
\bigcup_{i=1}^n z_i \left( \langle x_1 \rangle \times \langle x_2 \rangle\right)T_H=
\bigcup_{i=1}^n z_i \bigl( \langle (x_1,1) \rangle T_H\bigr)$. Which implies that
$C_{G \times G}(x) T_H$ is closed in $\PT(G \times G)$ as a finite union of closed sets.

Finally, since $G \times G$ is hereditarily conjugacy separable
by Lemma \ref{lem:dir_prod-hcs}, we can apply Lemma \ref{lem:h_c_s->c_s_for_sbgps}
to conclude that $T_H$ is conjugacy separable.
\end{proof}

We now turn to an analogue of Proposition \ref{prop:conj_sep->sep_quot}.

\begin{prop}\label{prop:fibre_conj_sep->rf_quot} Let $H$ be a finitely generated non-elementary
normal subgroup of a hyperbolic group $G$, with $E_G(H)=\{1\}$. If the fibre product $T_H \le G \times G$,
associated to $H$, is conjugacy separable, then $G/H$ is residually finite.
\end{prop}

\begin{proof} Since $H$ is non-elementary and $E_G(H)=\{1\}$,
there is an infinite order element $h_1 \in H$ such that $E_G(h_1)=\langle h_1 \rangle$
(see \cite[Lemma 3.4]{Olsh-G-sbgps}).
In particular, $C_G(h_1)=\langle h_1 \rangle \le H$ in $G$.

Consider any element $g_1 \in G \setminus H$, and set $g:=(g_1,1) \in G \times G$,
$h:=(h_1,h_1) \in G \times G$. Note that $g \notin T_H$ (as $g_1 \notin H$), therefore the elements
$h,f:=h^g \in H \times H \le T_H$
are not conjugate in $T_H$ because $C_{G\times G}(h)=\langle h_1 \rangle \times \langle h_1 \rangle \le T_H$.
The group $T_H$ is conjugacy separable by the assumptions, hence there exist a finite group
$F$ and a homomorphism $\zeta:T_H \to F$ such that $\zeta(f) \notin \zeta(h)^F$. Denote by
$\varphi:H \times H \to F$ the restriction of $\zeta$ to $H \times H$;
then $\varphi(f) \notin \varphi(h)^F$.

Since $H \times H$ is finitely generated, we can find a finite index subgroup
$N$ of $H \times H$ such that $N \le \ker(\varphi)$ and $N \lhd G \times G$. Define
$M:=(H \times H) /N$, and let $\alpha:H \times H \to M$ be the natural epimorphism.
Observe that, by definition, $\varphi$ factors through $\alpha$, hence
\begin{equation}
\label{eq:alpha_f-h}
\alpha(f) \notin \alpha(h)^M \mbox{  in $M$}.
\end{equation}

Define the map $\xi: G \to Aut(M)$ by $\xi(x_1)(y):=\alpha((x_1 y_1 x_1^{-1},y_2))$ for all $x_1 \in G$
and all $y \in M$, where $(y_1,y_2) \in H \times H$ is any element satisfying
$y=\alpha((y_1,y_2))$. Note that $\xi$ is a well-defined homomorphism because
$(x_1 y_1 x_1^{-1},y_2)=(x_1,1)(y_1,y_2) (x_1,1)^{-1}$ and $(x_1,1) N (x_1,1)^{-1}=N$.
Evidently $\xi(H) \le Inn(M)$,
therefore $\xi$ canonically gives rise to a homomorphism $\bar \xi: G/H \to Out(M)$,
where $Out(M):=Aut(M)/Inn(M)$ is the group of outer automorphisms of $M$.

Finally, we have $\xi(g_1)(\alpha(h))=\alpha(f)$, which together with
\eqref{eq:alpha_f-h} implies that $\xi(g_1) \notin Inn(M)$. Thus $\bar \xi(g_1H) \neq 1$ in
the finite group $Out(M)$. Since we started with an arbitrary element $g_1 \in G \setminus H$,
we can conclude that $G/H$ is residually finite.
\end{proof}

\begin{rem} A careful reader might note here that Propositions \ref{prop:sep->conj_sep} and
\ref{prop:conj_sep->sep_quot} could be proved similarly to Propositions \ref{prop:fibre_sep->conj_sep}
and \ref{prop:fibre_conj_sep->rf_quot}. However, we decided to present somewhat
different proofs for the former two statements, to make the argument in
Proposition  \ref{prop:sep->conj_sep} self-contained (independent of Lemma \ref{lem:h_c_s->c_s_for_sbgps}),
and the argument in Proposition \ref{prop:conj_sep->sep_quot} more conceptual and less technical.
\end{rem}

\begin{cor}\label{cor:non_rf_quot->non_cs}
Suppose that $G$ is a torsion-free residually finite hyperbolic group,
$H \lhd G$ is a finitely generated normal subgroup, and $T_H \le G \times G$ is the associated fibre product.
If $G/H$ is not residually finite, then $T_H$ is not conjugacy separable.
\end{cor}

\begin{proof} As we have already shown in the proof of Theorem \ref{thm:sep<->conj_sep},
if $H$ were elementary, then $G/H$ would be residually finite. Therefore $H$ is non-elementary.
And since $G$ is torsion-free, it does not contain any non-trivial finite subgroups. Hence
$E_G(H)=\{1\}$, and the claim follows from Proposition \ref{prop:fibre_conj_sep->rf_quot}.
\end{proof}

We finish this section with two remarks concerning the conjugacy problem in fibre products.
Let $G$ be a finitely generated group.
We will say that the \textit{membership problem to cyclic subgroups} (MPCS)
\textit{is uniformly decidable in} $G$, if there is an algorithm,
which takes on input any two elements $x,y \in G$
and determines whether or not $y \in \langle x \rangle$ in $G$. Clearly, the latter
property is a natural ``algorithmic'' analogue of cyclic subgroup separability
(indeed, for any finitely presented cyclic subgroup separable group $G$, MPCS
will be uniformly decidable by Mal'cev's result \cite{Malcev}).

We can now formulate the corresponding counterparts of Propositions \ref{prop:memb->conj}
and \ref{prop:conj->memb}.

\begin{prop} \label{prop:fibre_memb->conj} Let $H$ be a normal subgroup
of a torsion-free hyperbolic group $G$. If the quotient
$G/H$ has uniformly decidable MPCS, then the corresponding fibre product $T_H \le G \times G$
has solvable conjugacy problem.
\end{prop}

\begin{proof} We leave this as an exercise for the reader. It can be easily derived from the
proofs of Proposition \ref{prop:memb->conj} and Proposition \ref{prop:fibre_sep->conj_sep}.
\end{proof}

The proof of the next statement can be extracted from \cite[Thm. ${\rm A}'$, Lemma 3.3]{B-B-M-S}.

\begin{prop}\label{prop:fibre_conj->wp_for_quot}
Suppose that $G$ is a torsion-free hyperbolic group,
$H \lhd G$ is a finitely generated normal subgroup, and $T_H \le G \times G$ is the associated fibre product.
If $T_H$ has solvable  conjugacy problem then $G/H$ has solvable word problem.
\end{prop}

\begin{proof} Again this is an exercise in view of the proofs of Propositions \ref{prop:conj->memb}
and \ref{prop:fibre_conj_sep->rf_quot}.
\end{proof}

\section{Finitely presented examples}\label{sec:f_p_ex}
In this section we construct examples of finitely presented (non-hereditarily) conjugacy separable groups.
However, before proceeding we need one more auxiliary statement.

\begin{lemma}\label{lem:fibre_f_i} Suppose $N$ and $H$ are normal subgroups of a group $G$
such that $N \le H$ and $|H:N|< \infty$. Let $T_N, T_H \le G \times G$ be the corresponding
fibre products. Then $T_N \le T_H$ and $|T_H:T_N|=|H:N|$.
\end{lemma}

\begin{proof} The inclusion of $T_N$ in $T_H$ is an immediate consequence of the
definition of a fibre product. Set $P:=G/N$ and let $\eta:G\times G\to P\times P$
be the natural epimorphism with $\ker(\eta)=N \times N$.
Then $\eta(T_N)$ is the diagonal subgroup $D$ of $P \times P$, and
$\eta(T_H)=\eta(H \times H)D=\eta((H,1))D$, where $(H,1):=\{(h_1,1) \,|\,h_1 \in H\} \lhd G \times G$.

And since $\eta((H,1)) \cap D=\{1\}$ in $P \times P$, $\eta(T_H)$ is a semidirect product
of $\eta((H,1))$ and $D$. Therefore  $|\eta(T_H):\eta(T_N)|=|\eta(T_H):D|=|\eta((H,1))|=|H:N|$.
Recall that $\ker(\eta)=N \times N \le T_N \le  T_H$, hence $T_N$ and $T_H$ are the full $\eta$-preimages of
$\eta(T_N)$ and $\eta(T_H)$ in $G \times G$ respectively. Thus we can conclude that
$|T_H:T_N|=|\eta(T_H):\eta(T_N)|=|H:N|$.
\end{proof}

To produce the example, establishing Theorem \ref{thm:fp_ex},
we need to start with a short exact sequence of groups
\begin{equation}
\label{eq:short_exact-2} \{1\} \to F \to P \to Q \to \{1\},
\end{equation}
such that
\begin{itemize}
\item[(i)] $|F|=m<\infty$;
\item[(ii)] $Q$ is of type $F_3$;
\item[(iii)] $Q$ is cyclic subgroup separable;
\item[(iv)] $P$ is not residually finite.
\end{itemize}

Observe that the sequences given in
Examples \ref{ex:1} and \ref{ex:2} possess all these properties. Indeed, the groups $Q$,
arising there, are finite index subgroups of ${\rm Sp}(2n,\Z)$ or ${\rm SL}(3,\Z)$.
By a theorem A. Borel and J.-P. Serre \cite{Bor-Serre} such $Q$ is of type $F_n$ for every $n \in \N$.
On the other hand, $Q$ is cyclic subgroup separable because $Q \le {\rm GL}(m,\Z)$ for some $m \in \N$,
and ${\rm GL}(m,\Z)$ is cyclic subgroup separable (see \cite[Thm. 5, p. 61]{Segal}).

Now we can use Theorem \ref{thm:Rips-h_c_s} to find a hereditarily conjugacy separable torsion-free
hyperbolic group $G$ and a finitely generated normal subgroup $N\lhd G$ such that $G/N \cong P$. Let
$\psi:G \to P$ be the natural epimorphism with $\ker(\psi)=N$. Set $H:=\psi^{-1}(F) \lhd G$; then we have
the same commutative diagram \eqref{eq:cd} as before.

Let $T_N, T_H \le G \times G$ be the fibre
products associated to $N$ and $H$ respectively.  Note that $|H:N|=|F|<\infty$, hence $H$ is also finitely
generated. The group $G$ is finitely presented, as any hyperbolic group, and $G/H \cong Q$ is of type
$F_3$. Therefore Lemma \ref{lem:1-2-3} allows us to conclude that the group $T_H$ is finitely presented.

Finally, $T_H$ is conjugacy separable by Proposition \ref{prop:fibre_sep->conj_sep}, and $T_N$ is
not conjugacy separable by Corollary \ref{cor:non_rf_quot->non_cs}. And, according to Lemma \ref{lem:fibre_f_i},
$T_N \le T_H$ and $|T_H:T_N| =|H:N|=|F|=m$. Thus the group $T_H$ is a finitely presented
(non-hereditarily) conjugacy separable group.

To achieve the first additional claim of Theorem  \ref{thm:fp_ex},
first assume that $m$ is a prime number. Apply the above algorithm to find
the groups $F$, $P$, $Q$, $G$, $N$ and $H$ as before.
According to Theorem \ref{thm:Rips-h_c_s}, there is a finite index subgroup
$G_1 \le G$ with $G_1 \in \mathcal{VR}$. Observe that $P_1:=\psi(G_1)$ has finite index in $P$,
hence it cannot be residually finite, implying that $P_1 \cap F \neq \{1\}$. But since
$|F|=m$ is a prime, we can conclude that $F \le P_1$. Consequently, after setting
$N_1:=N \cap G_1$ and $H_1:=H \cap G_1$, we see that $N_1,H_1 \lhd G_1$ and $|H_1:N_1|=|F|=m$.

Since the class of right angled Artin groups is closed
under taking direct products, the group $G_1 \times G_1$ will be a subgroup of some right angled
Artin group. Let $T_{N_1},T_{H_1} \le G_1 \times G_1$
be the fibre products associated to $N_1$ and $H_1$ respectively.
As we showed above, $T_{H_1}$ is finitely presented and conjugacy separable,
$T_{N_1}$ is not conjugacy separable, and $|T_{H_1}:T_{N_1}|=m$.

Now, if $m$ is a composite number, write $m=m_1m_2 \cdots m_l$,
where $m_j$ is a prime for every $j=1,\dots,l$. Apply
the above construction to each $m_j$, finding a finitely presented conjugacy separable group $T_j$,
which is a subgroup of some right angled Artin group $A_j$, and contains a non-(conjugacy separable)
subgroup $S_j$ of index $m_j$. Define the direct products
$T:= \prod_{j=1}^l T_j$ and $S:= \prod_{j=1}^l S_j \le T$. Then
$T$ is a finitely presented subgroup of the right angled Artin group $A:= \prod_{j=1}^l A_j$
and $|T:S|=m_1 \cdots m_l=m$. It is easy to see that direct products of conjugacy separable groups are
conjugacy separable, hence $T$ is conjugacy separable. On the other hand, $S$ cannot be
conjugacy separable, because $S_1$ is not conjugacy separable (and any retract of a
conjugacy separable group is itself conjugacy separable -- see \cite[Lemma 9.3]{Min-RAAG}).

We thus obtain

\begin{cor} \label{cor:3} There exists finitely presented subgroups of right angled Artin groups
that are conjugacy separable but not hereditarily conjugacy separable.
\end{cor}

A group $G$, generated by a set $X$, is said to have \emph{solvable order problem} if there is an algorithm which, given any word $W$ over $X^{\pm 1}$, 
outputs $n \in \N$ as long as the order of the element of $G$ represented by $W$ is $n$ or $0$ if this element has infinite order. Note that if $G$ has solvable word 
problem and is virtually torsion-free (so that there is a bound on the possible finite orders of its elements), then it has solvable order problem.

It is not difficult to show that if a finitely presented group $Q$ has uniformly decidable
MPCS and solvable order problem, then for any finite group $F$, any extension $P$ of $F$ by $Q$ also has uniformly decidable MPCS.
Therefore, Proposition \ref{prop:fibre_memb->conj} (together with Mal'cev's theorem
mentioned right before it and Selberg's Lemma \cite[Cor. 4.8]{Wehr-book}, stating that finitely generated subgroups of ${\rm GL}(m,\Z)$ are virtually torsion-free) 
yields the second additional claim of Theorem \ref{thm:fp_ex}:

\begin{cor} \label{cor:4} Both of the groups $T$ and $S$ above have solvable conjugacy problem.
\end{cor}

\end{document}